\documentclass[11pt]{amsart}
\title{Graded AR Sequences and the Huneke-Wiegand Conjecture
} 
\author{Robert Roy
}

\usepackage{amssymb, amscd, amsmath, float, graphicx, longtable,color,
  enumerate, fullpage
}
\usepackage[textwidth=2cm,textsize=small]{todonotes}
\usepackage{amsfonts}
\usepackage{hyperref}
\usepackage{verbatim}
\usepackage{mathrsfs}
\usepackage{psfrag}
\usepackage{fouriernc}\DeclareMathAlphabet{\mathcal}{OMS}{cmsy}{m}{n}
\usepackage[all]{xy}
\SelectTips{eu}{} 
\xyoption{curve}
\date{}
\usepackage{amssymb}
\usepackage{fullpage}

\DeclareMathOperator{\length}{length}

\DeclareMathOperator{\CM}{CM}

\DeclareMathOperator{\End}{End}
\DeclareMathOperator{\edim}{edim}
\DeclareMathOperator{\Hom}{Hom}
\DeclareMathOperator{\Ext}{Ext}

\DeclareMathOperator{\Free}{Free}

\renewcommand{\to}{\longrightarrow}

\newcommand{\NN}{\mathbb{N}}

\renewcommand{\hbar}{\overline{h}}

\newcommand{\stHom}{\underline{\Hom}}

\newcommand{\m}{\mathfrak{m}}

\newcommand{\into}{\hookrightarrow}

\newcommand{\p}{\mathfrak{p}}

\theoremstyle{plain}
\newtheorem{theorem}{Theorem}

\newtheorem{proposition}[theorem]{Proposition}

\newtheorem{lemma}[theorem]{Lemma}

\newtheorem*{theorem*}{Theorem}
\newtheorem*{prop*}{Proposition}

\theoremstyle{definition}
\newtheorem{conjecture}[theorem]{Conjecture}

\newtheorem{definition}[theorem]{Definition}

\newtheorem{remark}[theorem]{Remark}
\newtheorem{notation}[theorem]{Notation}

\numberwithin{theorem}{section}
\numberwithin{equation}{section}

\begin{document}
\large
\maketitle

\bibliographystyle{amsplain}

\begin{abstract}
We let $R=\bigoplus_{i \ge 0} R_i$ be a one-dimensional graded complete intersection, finitely-generated over  $R_0$ an infinite field, satisfying certain degree conditions which are satisfied whenever $R$ is a numerical semigroup ring of embedding dimension at least three. We show that a broad class of graded maximal Cohen-Macaulay $R$-modules $M$ satisfies the Huneke-Wiegand Conjecture, namely when there exists an Auslander-Reiten sequence ending in $M$ whose middle term  has at least two nonfree direct summands. 
\end{abstract}

\section{Introduction} 

%Notation, which will make more sense when the rings have been defined: We will use $\Omega_R(\,)$ or $\Omega(\,)$ to indicate the syzygy of $(\,)$, that is, the kernel of a minimal surjection onto $(\,)$ by a free $R$-module.  Write  $\stHom_R(X,Y)$ for the quotient of $\Hom_R(X,Y)$ by those morphisms which factor through free $R$-modules. Denote by  $L_p(R)$  the category of graded maximal Cohen-Macaulay $R$-modules $M$ such that $M_\p$ is free for all nonmaximal prime ideals $\p$. \vspace{1pc}

The following is called the Huneke-Wiegand Conjecture in, e.g., \cite{Garcia-Sanchez-Leamer} and \cite{Goto:HW}.

\begin{conjecture} (\cite{Huneke1994}) Let $D$ be a Gorenstein local domain of dimension one and $M$ a nonzero finitely-generated torsionfree $D$-module, that is not free. Then $M \otimes_D M^*$ has a nonzero torsion submodule. \end{conjecture}

\noindent As shown in \cite[Theorem 5.9]{Huneke-Jorgensen}, the above condition on $M \otimes_D M^*$ may be replaced by the condition that $\Ext_D^1(M,M) \neq 0$.

\begin{notation} If $R$ is a Gorenstein ring of dimension one, and $M$ is a nonfree  finitely-generated maximal Cohen-Macaulay module, we will say that $M$ ``satisfies the Huneke-Wiegand Conjecture" if $\Ext_R^1(M,M) \neq 0$.
 %In turn, this is equivalent to $\stHom_D(\Omega_D M, M) \neq 0$ (Lemma~\ref{lem:in-turn}).  \vspace{1pc}
\end{notation}

This paper is structured as follows. In Section~\ref{sec:AR} we define some general notions, and state lemmas in the setting of a graded one-dimensional Gorenstein ring, whose use in Section~\ref{sec:HW} will usually be explicitly mentioned, so these lemmas need not be read beforehand.  In Section~\ref{sec:HW}, we let $R$ be a one-dimensional graded complete intersection satisfying a certain degree condition  (Definition~\ref{defn:star}), and  prove (Theorem~\ref{main-thm}) that $M$ satisfies the Huneke-Wiegand Conjecture if $M$ is any module in a broad class (see Remark~\ref{rmk:broad-class}), namely when there exists an Auslander-Reiten sequence $s(M)$  ending in $M$  such that the  middle term of $s(M)$ has at least two nonfree direct  summands.  In Section~\ref{sec:ns} we  show that the degree condition is satisfied by all numerical semigroup rings of embedding dimension at least three.

%Given a  Auslander-Reiten sequence  of graded modules, 
%
%\begin{equation}\label {eqn:AR seq} \xymatrix{ 0 \ar[r] & N \ar[r]^p & X \ar[r]^q & M \ar[r] & 0 }\end{equation} 
%we prove that the Huneke-Wiegand conjecture holds for $M$, provided that $X$ has at least two nonfree direct summands. 

\section{Background: Auslander-Reiten Theory}\label{sec:AR}

Throughout, $R$ will be a commutative, one-dimensional, Gorenstein,  graded ring $R=\bigoplus_{i \ge 0}R_i$ graded over the nonnegative integers, where $R_0=k$ is a field and $R$ is a finitely-generated algebra over $k$. Denote the graded maximal ideal $\bigoplus_{i \ge 1}R_i$ by $\m$. For any graded $R$-module $M$,   $\Omega M$ will denote the graded syzygy of $M$, i.e., the kernel of a minimal graded surjection onto $M$ by a free $R$-module. We use $M^*$ to denote $\Hom_R( M,R)$. 

We let $\CM(R)$ denote the  category of graded, finitely-generated, maximal Cohen-Macaulay modules (and all $R$-linear homomorphisms between such), and let $L_p(R)$ denote the full subcategory $\{M \in \CM(R)| M_{\p} \text{ is } R_{\p}\text{-free for all nonmaximal prime ideals } \p \subset R\}$. Let $\CM(R)_0$ denote the category whose objects are those of $\CM(R)$ and whose morphisms are the degree zero maps. We say that a module $M \in \CM(R)$ without free direct summands is \emph{periodic} if $M \cong \Omega^2 M$, up to a shift in the grading. 
% Equivalently, the modules in the free resolution of $M$ have bounded rank \cite[\S 4]{Eisenbud:1980}.

\begin{notation} For $R$-modules $X$ and $Y$, we let $\stHom_R(X, Y)$ denote  the quotient of $\Hom_R(X, Y)$ by those maps which factor through free $R$-modules. \end{notation}

\begin{lemma} \label{lem:in-turn} Let $X, Y \in \CM(R)$. Then $\Ext^1_R(X,Y) \cong \stHom_R(\Omega X, Y)$. \end{lemma}
\begin{proof} Start with the short exact sequence \begin{equation}\label{lets-start}\xymatrix{0\ar[r] & \Omega X \ar[r]^-i & F \ar[r]^-j & X \ar[r]& 0,}\end{equation} where $F$ is a free module. Since $(\,)^*$ is a duality on $\CM(R)$, $i^*$ is epi. Therefore any map from a free module to $(\Omega X)^*$ factors through $i^*$. Then the duality implies that any map from $\Omega X $ to a free module factors through $i$. Therefore we have a right exact sequence \begin{center} \xymatrix@C+1pc{\Hom_R(F,Y) \ar[r]^-{\Hom(i,Y)} & \Hom_R(\Omega X, Y) \ar[r]& \stHom_R( \Omega X, Y) \ar[r]& 0.}\end{center}
 On the other hand, applying $\Hom_R(\_, Y)$ to~\eqref{lets-start} yields \begin{center} \xymatrix@C+1pc{\Hom_R(F,Y) \ar[r]^-{\Hom(i,Y)} &\Hom_R(\Omega X, Y) \ar[r] & \Ext^1_R(X, Y) \ar[r] & \Ext^1_R(F, Y) =0,}\end{center} and the lemma follows.
\end{proof}

\begin{notation} The symbol $a(R)$ stands for the integer called the \emph{a-invariant} of $R$, a formula for which is given in~\eqref{a(R)-eqn}. \end{notation}

\begin{lemma}\label{lem:Rmk} There exists a short exact sequence in $\CM(R)_0$ of the form \begin{center}$0\to R(a) \to \m^*(a) \to k\to 0$, \end{center} 
where $a=a(R)$.\end{lemma}
\begin{proof} As $R$ is Gorenstein of dimension one, we have $\Ext^1_R(k,R(a)) =k$. We can obtain the desired sequence by  applying $\Hom_R(\_,R(a))$ to the short exact sequence $0 \to \m \to R \to k \to 0$. \end{proof}

A morphism $f \colon X \to Y$ in $\CM(R)_0$ is called \emph{irreducible}  if (1) $f$ is neither a split monomorphism nor a split epimorphism, and (2) given any pair of morphisms $g$ and $ h$ in $\CM(R)_0$  satisfying $f=gh$, either $g$ is a split epimorphism or $h$ is a split monomorphism. 

 Let $M$ be a nonfree indecomposable in $L_p(R)$. Then  \cite[Theorem 3]{Auslander-Reiten:gradedCM} the category $\CM(R)_0$ admits  an Auslander-Reiten (AR) sequence ending in $M$. That is, there exists a short exact sequence 
 \begin{equation}\label {eqn:AR seq} \xymatrix{ 0 \ar[r] & N \ar[r]^f & X \ar[r]^g & M \ar[r] & 0 }\end{equation}  
in $\CM(R)_0$ such that $N$ is indecomposable and the following property is satisfied: Any map $L \to M$ in $\CM(R)_0$ which is not a split epimorphism factors through $g$ (equivalently, any map $N \to L$ in $\CM(R)_0$ which is not a split monomorphism factors through $f$). In this paper $R$ is Gorenstein, and therefore (1) there is also an AR sequence \emph{beginning} at $M$, and (2) the modules in~\eqref{eqn:AR seq} all lie in $L_p(R)$ (to see this, use Lemma~\ref{tau-formula}). 

\begin{definition} \label{defn:tau} Given an AR sequence~\eqref{eqn:AR seq}, $N$ is called the Auslander-Reiten translate of $M$, written $\tau M$. One may equivalently write $\tau^{-1} N=M$. \end{definition}

\begin{lemma} \label{every-irreducible-map} Let  $\xymatrix{0 \ar[r] & \tau M \ar[r]^-{f} & X \ar[r]^-g &  M \ar[r] & 0}$  be an AR sequence in $\CM(R)_0$. 
Then, given any $Y \in \CM(R)$, a degree zero map $h \colon \tau M \to Y$ is irreducible if and only if there exists a split epimorphism $p \in \Hom_R(X,Y)_0$ such that $h=pf$. Dually, a degree zero map $h' \colon Y' \to M$ is irreducible if and only if there exists a split monomorphism $\iota \in \Hom_R(Y,X)_0$ such that $h'=g \iota$.
\end{lemma}

\begin{proof} Cf. \cite[Lemma 2.13]{Yoshino:book}, \cite[Ch. V, Theorem 5.3]{AuslanderReitenSmalo}. \end{proof}

\begin{lemma} \label{irreducible-maps-are-cancellative} (\cite[Lemma 4.1.8]{THESIS}) If $f \colon M \to N$ is an irreducible map in $\CM(R)_0$, then $f$ must be either a monomorphism or an epimorphism. \end{lemma}

\begin{notation} \label{row-and-column-matrices} Several short exact sequences in this paper will include a written direct-sum-decomposition (not necessarily into indecomposables) of the middle term, for example 

\begin{center}$\xymatrix{0 \ar[r] & X \ar[r]^-f & Y_1 \oplus Y_2 \ar[r]^-g & Z \ar[r] & 0}$.
\end{center}
 Given such a diagram, if we write $f=[f_1,f_2]^T$ and $g=[g_1,g_2]$, this means that $g$ restricted to $Y_i$ is $g_i$, and similarly $f$ induces maps $f_1 \colon X \to Y_1$ and $f_2 \colon X \to Y_2$ with respect to the direct sum decomposition $Y_1 \oplus Y_2$. \end{notation}

\begin{lemma} \label{epi-pairs} If $\xymatrix{0 \ar[r] & X \ar[r]^-{[f_1, f_2]^T} & Y_1 \oplus Y_2 \ar[r]^-{[g_1,g_2]} & Z \ar[r] & 0}$ is any short exact sequence of abelian groups, then $f_1$ is an epimorphism if and only if $g_2$ is an epimorphism. If it is an AR sequence in $\CM(R)_0$, then each $f_i$ and $g_i$ is irreducible, and: $f_1$ is a monomorphism if and only if $g_2$ is a monomorphism. \end{lemma}

\begin{proof} The first sentence is straightforward, and the second  follows from Lemmas~\ref{every-irreducible-map} and~\ref{irreducible-maps-are-cancellative}. \end{proof}

\begin{notation} If $M$ is a graded $R$-module, and $n \in \mathbb{Z}$, then the graded shift $M(n)$ is the graded $R$-module given by $M(n)_i=M_{n+i}, \forall i \in \mathbb{Z}$. \end{notation}

\begin{lemma}\label{tau-formula} For any indecomposable nonfree $M 
\in L_p(R)$,  we have $\tau M=\Omega M (a)$, where $a=a(R)$ is the $a$-invariant of $R$. \end{lemma}

\begin{proof} Cf. \cite[Proposition 3.11]{Yoshino:book} or the proof of \cite[Theorem 3]{Auslander-Reiten:gradedCM}. \end{proof}

\begin{notation} \label{ntn:syz-of-a-map} Let $f\colon M \to N$ be a morphism in $\CM(R)_0$. By extending $f$ to a map between the minimal free resolutions of $M$ and $N$, we get induced morphisms $\Omega^n f \colon  \Omega^n M \to \Omega^n N \in \CM(R)_0$, for each integer $n$. (Of course, these are not quite uniquely determined.) Likewise, we have morphisms $\tau^n f \colon \Omega^n M (a(R)) \to \Omega^n N (a(R))$. \end{notation}

\begin{lemma} \label{lem:syz-of-irreducible} (\cite[Theorem 3.1]{Puth},\cite[Lemma 4.1.7]{THESIS})  Let $f\colon M \to N$ be a morphism in $\CM(R)_0$, and assume $M$ and $N$ contain no free direct summands. If $f$ is irreducible, then so is any choice of $\Omega^n f$, for all $n \in \mathbb{Z}$.\end{lemma}

\begin{lemma} \label{stablezero-in-radical} Let $M$ and $N$ be finitely-generated, graded $R$-modules, and assume $M$ has no free direct summand. If $f \colon M \to N$ factors through a free $R$-module, then $f(M) \subseteq \m N$. \end{lemma}

\begin{proof} We may  assume $N$ is free. As $f$ is necessarily a sum of graded maps (cf. \cite[Exercise 1.5.19 (f)]{BH}), we may assume $f$ is homogeneous. Now if $m \in M$ is homogeneous, and $f(m) \notin \m N$, then $f(m)$ can be extended to a basis of $N$, and we may obtain a surjection $M \to R f(m)$, contradicting our assumption that $M$ has no free direct summand.
\end{proof}

\begin{lemma} Let $f \colon M \to N$ be a morphism in $\CM(R)_0$, and let $n$ be an integer. If $g$ and $h \in \Hom_R(\Omega^n M, \Omega^n N)$ are two choices for $\Omega^n f$, and $g$ is epi, then so is $h$. \end{lemma}

\begin{proof} For $n\ge 0$, it is well known that $g-h$ must factor through a free module, and it is easy to see that this still holds for $n<0$ by Gorenstein duality. Therefore $(g-h)(\Omega^nM) \subseteq \m \Omega^n N$ by Lemma~\ref{stablezero-in-radical}. The assertion now follows from Nakayama's Lemma. \end{proof}

\begin{lemma}\cite[Lemma 4.1.13]{THESIS} \label{lem:counting-k-dim} Let $\xymatrix {0 \ar[r] & X \ar[r] & Y \ar[r]^g & Z \ar[r] &0}$ be a short exact sequence in $\CM(R)_0$. Then $\Omega g$ is an epimorphism if and only if $\m X=X \cap \m Y$. \end{lemma}

\begin{lemma} \label{lem:A-is-an-ideal} (cf. \cite[Lemma 2.1]{GZ}) Let  $\xymatrix {0 \ar[r] & X \ar[r]^f & Y \ar[r]^g & Z \ar[r] &0}$ 
be a short exact sequence in $\CM(R)_0$ with $g$ irreducible, and suppose that $\Omega g$ is a monomorphism. Then there exist graded maps $i \colon X \to \m^*$ and $j \colon \m^* \to Y$ in $\CM (R)$ such that $f=j i$; in particular, $X$ is isomorphic to a graded ideal of $R$, up to a graded shift. \end{lemma}

\begin{proof} It is part of the general (Auslander-Reiten) folklore that if $f' \colon X \to Y'$ is any map in $\CM(R)_0$, then either $f$ factors through $f'$ or $f'$ factors through $f$ (in $\CM(R)_0$). To see this, assume that $f'$ does not factor through $f$. This says that the pushout of $\xymatrix {0 \ar[r] & X \ar[r]^f & Y \ar[r]^g & Z \ar[r] &0}$ by $f'$ does not split. Therefore, the irreducibility of  $g$ implies that the middle map in the diagram \begin{center}
$\xymatrix {0 \ar[r] & X \ar[d]_{f'} \ar[r]^f & Y  \ar[d] \ar[r]^g & Z \ar@{=}[d] \ar[r] &0\\
0 \ar[r] & Y' \ar[r] & W \ar[r] & Z \ar[r] &0}$ \end{center}
is a split monomorphism, and it follows that $f$ factors through $f'$.

  From the short exact sequence $0 \to R(a) \to \m^*(a) \to k \to 0$ (Lemma~\ref{lem:Rmk}), we obtain a commutative square \begin{equation} \label{contrary-epis} 
\xymatrix {\Hom_R(Y,\m^*(a)) \ar[r] \ar[d] & \Hom_R(Y,k) \ar[d]\\
\Hom_R(X,\m^*(a)) \ar[r] & \Hom_R(X,k)}\end{equation} 
where the horizontal maps are surjective since $\Ext^1_R(Y,R)=\Ext^1_R(X,R)=0$ (as $R$ is one-dimensional Gorenstein). From Lemma~\ref{lem:counting-k-dim} it follows that $\dim_k(Y \otimes_R k)< \dim_k(X \otimes_R k)+\dim_k(Z \otimes_R k)$, i.e.,  $\dim_k(\Hom_R(Y,k))< \dim_k(\Hom_R(X,k))+\dim_k(\Hom_R(Z,k))$, and therefore the map  $\Hom_R(f,k) \colon \Hom_R(Y,k) \to \Hom_R(X,k)$ is not an epimorphism. Then, in turn, the map  $\Hom_R(f,m^*(a)) \colon \Hom_R(Y,m^*(a)) \to \Hom_R(X,m^*(a))$ is not epi (because the horizontal maps in~\eqref{contrary-epis} \emph{are} epi).  Therefore, we may pick  $f' \in \Hom_R( X , \m^*)$ which does not factor through $f$. We may also assume $f$ is homogeneous of some degree, cf. \cite[Exercise 1.5.19 (f)]{BH}. It then follows from our first sentence in this proof that $f=jf'$ for some homogeneous $j \colon \m^* \to Y$. We know that $f' \colon X \to \m^*$ is mono, because $f$ is mono. Lastly, note that there exists a  monomorphism $X \into R$, because there exists a monomorphism  $\m^* \into R$ (since $\m^*$ is a finitely-generated submodule of $R[\text{nonzerodivisors}]^{-1}$).
\end{proof}

\begin{definition} \label{defn:perfect} If $X$ and $Y$ are modules in $\CM(R)$ having no free direct summands, we say that an irreducible map $g \colon X \to Y$ is \emph{eventually} $\Omega$-\emph{perfect} if either $\Omega^n g$ is epi for all large $n$, or $\Omega^n g$ is mono for all large $n$. \end{definition}

\begin{lemma} \label{lem:periodic-ideal} (cf. \cite[Proposition 2.4]{GZ}) Assume that $R$ is a complete intersection, and let $\xymatrix {0 \ar[r] &  X \ar[r] &  Y \ar[r]^-{g} &  Z \ar[r] &0}$ be a short exact sequence in $\CM(R)_0$, such that $g$ is irreducible and  not  eventually $\Omega$-perfect. Then, for all $n \ge 0$ such that  $\Omega^n g$ is epi and $\Omega^{n+1} g$ is mono, we have that $\ker(\Omega^n g)$ is isomorphic to a periodic ideal. \end{lemma}

\begin{proof} For each $n \ge 0$ we can apply the Horseshoe Lemma to obtain a short exact sequence $\xymatrix {0 \ar[r] & \Omega^n X \ar[r] ^-{[f_1,f_2]^T}& F^n \oplus \Omega^n Y \ar[r]^-{[\xi,\Omega^n g]} & \Omega^n Z \ar[r] &0}$ for some free module $F^n$, and $\xi \in \Hom_R(F^n, \Omega^n Z)_0$. If $\Omega^n g$ is surjective, then so is $f_1$ (Lemma~\ref{epi-pairs}) and this implies $F^n=0$, since  $\Omega^n X$ cannot have a free direct summand. So if  $\Omega^{n+1} g$ is mono while  $\Omega^n g$ is surjective, then  \begin{center}$\Omega^n X \cong \ker(\Omega^n g)$  is isomorphic to an ideal,\end{center} by Lemma~\ref{lem:A-is-an-ideal}. Since $\Omega^n X$ is thus isomorphic to an ideal for infinitely many $n \ge 0$, it follows that the free modules in a minimal resolution of $X$ have bounded rank (cf. \cite[Lemma 4.1.17]{THESIS}). Therefore $X$ is eventually periodic, by \cite[Theorem 4.1]{Eisenbud:1980}. As $X \in \CM(R)$ and $R$ is Gorenstein,  $X$ is in fact already periodic.
\end{proof}

\section{Huneke-Wiegand Conjecture} \label{sec:HW}

 Let $S=k[x_1, \dots, x_e]=\bigoplus_{i \ge 0} S_i$ be a graded polynomial ring over an infinite field $k$, where each $x_i$ is homogenenous (of any positive degree),  $k=S_0$, and $e \ge 2$. Let $f_1, \dots, f_{e-1} \in S$ be a regular sequence of homogeneous polynomials,  and let $R=S/(f_1,\dots, f_{e-1})$. We assume that $e=\edim R$, the embedding dimension of $R$ (i.e., each $f_i \in (x_1,\dots, x_e)^2)$).  Let $a_i=\deg x_i$, and $d_i=\deg f_i$. Then the $a$-invariant of $R$ is  (\cite[3.6.14-15]{BH}) \begin{equation}\label{a(R)-eqn} a(R)=\sum_{i=1}^{e-1} d_i -\sum_{i=1}^e a_i. \end{equation}  

\begin{definition} \label{defn:star} Let $d_{max}=\max\{d_1,\dots,d_{e-1}\}$. We will say that $R$ ``satisfies Condition (*)" if \begin{center} $a(R)>d_{max}/2$. \end{center} \end{definition}

\noindent In Section~\ref{sec:ns} we show that Condition (*) is satisfied by all  complete intersection numerical semigroup rings of embedding dimension at least three.

\begin{definition} We will call an indecomposable  module $M \in L_p(R)$ \emph{elevated} if  \begin{center} $\min\{i| (\tau M)_i \neq 0\} <\min\{i| M_i \neq 0\}$.
\end{center} (Recall that $\tau$ denotes the AR translate; see Definition~\ref{defn:tau} and Lemma~\ref{tau-formula}.)\end{definition}

Notice that if $M$ is elevated, then there exists no monomorphism $\tau M \into M$ in $\CM(R)_0$.

\begin{lemma}\label{tau-is-lower} Assume $R$ satisfies Condition (*), and let $M$ be an indecomposable nonfree module in $L_p(R)$. Then there exists $n_0$, depending on $M$, such that for all $n \ge n_0$, either $\Omega^n M$ is elevated or $\Omega^{n+1} M$ is elevated (or both).\end{lemma}

\begin{proof}  Let $({\bf F},\partial) \colon \dots \xymatrix{  \ar[r]^-\partial & F_1 \ar[r]^-\partial & F_0}$ be a graded $R$-free resolution of $M$.  Following Eisenbud's construction of cohomology operators \cite{Eisenbud:1980}, take a graded lifting $({\bf \tilde{F}},\tilde{\partial})$ of $({\bf F},\partial)$ to free $S$-modules, that is, $({\bf \tilde{F}},\tilde{\partial})$ is a graded sequence of free $S$-modules  such that $\tilde{\partial} \otimes_S R=\partial$. Then for $j \in \{1,\dots, e-1\}$, and all $n$, we can choose graded maps $t_j \colon \tilde{F}_{n+2} \to \tilde{F}_{n}$ such that \begin{center}$\tilde{\partial}^2=\sum_j f_j t_j$, and $\deg t_j =-d_j$ for each $j \in \{1,\dots, e-1\}$.\end{center} Now let $\hat{R}=\prod_{i\ge 0} R_i$ (the completion of $R$ with respect to $\m$) and consider the resolution of free $\hat{R}$-modules $({\bf \hat{F}},\hat{\partial})$ induced by $\partial$, as well as the maps $\hat{t_j}  \colon \hat{F}_{n+2} \to \hat{F}_{n}$ induced by the $t_j$'s. From the proof of \cite[Theorem 3.1]{Eisenbud:1980}, there exist $g_1,\dots, g_e \in \hat{R}$ such that $\hat{t_1}+\sum_{j=2}^{e-1} g_j \hat{t_j} \colon \hat{F}_{n+2} \to \hat{F}_{n}$ is an epimorphism for all large $n$. 

It follows that, for large $n$, $\min\{i| (F_{n+2})_i \neq 0\} - \min\{i| (F_{n})_i \neq 0\} \le d_{max}$. In other words,  $\min\{i| (\Omega^{n+2} M)_i \neq 0\}- \min\{i| (\Omega^{n}M)_i \neq 0\} \le d_{max}$, which implies that either $\min\{i| (\Omega^{n+2} M)_i \neq 0\} -\min\{i| (\Omega^{n+1} M)_i \neq 0\} \le d_{max}/2$ or $\min\{i| (\Omega^{n+1} M)_i \neq 0\} -\min\{i| (\Omega^{n} M)_i \neq 0\} \le d_{max}/2$. The result now follows from Condition (*), since $\tau^n M =\Omega^n M (n \cdot a(R))$, by Lemma~\ref{tau-formula}. \end{proof}

\begin{definition} \label{defn:alpha} Given an AR sequence $0 \to \tau M \to X \to M \to 0$, and writing $X=\bigoplus_i X^i$ as a direct sum of indecomposable modules $X^i$, define $\alpha(M)$ to be the number of (not necessarily nonisomorphic) summands $X^i$ which are nonfree. \end{definition}

\begin{remark} \label{rmk:broad-class} Let $X$ and $Y$ be nonfree  indecomposables in $L_p(R)$, such that $\alpha(X)$ and $\alpha(Y)$ are both equal to 1. For the sake of simplicity, assume also that $X$ is not a direct summand of $\m$ (so that there exists no irreducible map $X \to R$), and $Y$ is not a direct summand of $\m^*$  (so that there exists no irreducible map $R \to Y$). Then there exists no irreducible morphism $X \to Y$ (i.e., $X$ and $Y$ are not adjacent in the stable AR quiver). Indeed, any irreducible map $X \to X'$ is mono and any irreducible map $Y' \to Y$ is epi, by Lemma~\ref{every-irreducible-map}. Thus the  nonfree indecomposables $M \in L_p(R)$ having $\alpha(M) \ge 2$ form a broad class.
\end{remark}
\begin{lemma} \label{elevation-epi} Assume $R$ satisfies Condition (*), and let $M$ be a nonfree indecomposable in $L_p(R)$, with $\alpha(M) \ge 2$. Then there exists a nonfree indecomposable $X \in \CM(R)$ and   an irreducible morphism $p \colon \tau M \to X$ in $\CM(R)_0$, such that the set \begin{center}$\mathcal{N}: =\{n \ge 0| \Omega^n p  \text{ is epi and } \Omega^n  M \text { is elevated}\}$\end{center} is infinite. \end{lemma}

\begin{proof} Let $\xymatrix@C+1.5pc{ 0 \ar[r] & \tau M \ar[r]^-{[f_1,\dots,f_{l+1}]^T} & \bigoplus_{i=1}^l X^i \oplus F \ar[r]^-{[g_1,\dots,g_{l+1}] } & M \ar[r]& 0}$ be the AR sequence in $\CM(R)_0$ ending in $M$, where $X^1 , X^2, \dots, X^l$ are nonfree indecomposables in $L_p(R)$, and $F$ is a (possibly zero) free module.  Then it follows from  Lemmas~\ref{every-irreducible-map} and~\ref{lem:syz-of-irreducible} that the AR sequence ending in $\tau^nM$ has the form  \begin{equation}\label{eqn:f'g'} \xymatrix{ 0 \ar[r] & \tau^{n+1} M \ar[r]^-{f'} & \bigoplus_{i=1}^l \tau^n(X^i) \oplus F' \ar[r]^-{g'} & \tau^n M \ar[r]& 0}\end{equation}
 where $F'$ is free and $f'=[\tau^n f_1,\dots, \tau^n f_l, \xi]^T$, for some $\xi \colon \tau^{n+1} M \to F'$. 

Certainly $\tau^n M$ is elevated for infinitely many $n \ge 0$, by Lemma~\ref{tau-is-lower}. We claim that for each such $n$, $\tau^n f_i$ is epi for  some $i \in \{1,\dots, l\}$. Note that, given the AR sequence~\eqref{eqn:f'g'}, the map $\xi \colon \tau^n M \to F'$ must be mono, since otherwise it would be epi (Lemma~\ref{irreducible-maps-are-cancellative}) and therefore  split, contradicting the fact that it is irreducible. So if we write $g'=[g'_1,\dots, g'_{l+1}]$, and if  $\tau^n f_i$ is mono for each $i\in \{1,\dots, l\}$, then it follows from Lemma~\ref{epi-pairs} that each of $g'_1,\dots g'_{l+1}$ is mono. Then $g'_1 \circ ( \tau^n f_1)$, for example, is a degree zero monomorphism, which contradicts that $\tau^nM$ is elevated. Thus the claim is thus proved, and the lemma follows (by applying a version of the ``pigeonhole principle").
 \end{proof}

\begin{lemma}\label{why-edim-ge-3} If $\edim R \ge 3$, then $\m^*$ is not periodic.\end{lemma}

\begin{proof} See  \cite[Theorem 8.1.2]{Avramov:6lectures}. \end{proof}

\begin{theorem}\label{main-thm}  Assume that either $R$ is isomorphic to a numerical semigroup ring $k[t^{a_1},\dots, t^{a_e}]$, or that $\edim R \ge 3$ and $R$ satisfies Condition (*). If $M \in L_p(R)$ is a nonfree indecomposable  with $\alpha(M) \ge 2$, then $\stHom_R(\tau M,M)_0 \neq 0$, and $M$ satisfies the Huneke-Wiegand Conjecture. \end{theorem}

\begin{proof}
If $\edim R <3$, and $R$ is a domain (for example, a numerical semigroup ring), then $M$ satisfies the Huneke-Wiegand Conjecture by \cite[Theorem 3.7]{Huneke1994}. If $R$ is a numerical semigroup ring with  $\edim R \ge 3$, then we will see in Proposition \ref{prop:ns} that $R$ satisfies Condition (*). So for this proof, assume that $\edim R \ge 3$ and $R$ satisfies Condition (*). Now assume,   to the contrary, that $\stHom_R(\tau M,M)_0 =0$. Take $p \colon \tau M \to X$ and $\mathcal{N}$ as in Lemma~\ref{elevation-epi}.

First we set about proving the following claim: $\tau^{n+1} p$ is epi for all sufficiently large $n \in \mathcal{N}$. 

 If $p$ is eventually $\Omega$-perfect then the claim holds (by choice of $p$), so assume otherwise. Let $n \in \mathcal{N}$, i.e.,  $\tau^n p \colon \tau^{n+1} M \to \tau^n X$ is epi and $\tau^n M$ is elevated. Let us first observe that: \begin{equation}\label{eqn:first-observe} \text{There exists an irreducible mono }\iota \colon \tau^n X \to \tau^n M \text{ in } \CM(R)_0.\end{equation} Indeed, there exists an irreducible map $\tau^n X \to \tau^n M$ by Lemma~\ref{every-irreducible-map}, and if it were epi then the composition with the epimorphism $\tau^n p$ would (by Lemma~\ref{stablezero-in-radical}) give a nonzero element of $\stHom_R(\tau^{n+1} M, \tau^n M)_0 \cong \stHom_R(\tau M,M)_0$. (Indeed,  $\stHom_R(\Omega^{n+1} M, \Omega^n M) \cong \stHom_R(\Omega M,M)$ for all $n \in \mathbb{Z}$ since $R$ is Gorenstein.) 

Now, suppose $\tau^{n+1} p$ is mono.  Letting $K=\ker ( \tau^{n} p)$, we know $K$ is isomorphic to a periodic ideal, by Lemma~\ref{lem:periodic-ideal}. Letting $\kappa$ denote the inclusion map  $\kappa \colon K \into \tau^{n+1} M$,  we have $\kappa=\kappa'' \kappa'$ for some graded maps $\kappa' \colon K \to \m^*$ and $\kappa'' \colon \m^* \to \tau^{n+1} M$, by Lemma~\ref{lem:A-is-an-ideal}. Note that $K \cong K^{**} \cong \Hom_R(K^*,\m)$ since $K$ lies in $ \CM(R)$ and has no free direct summand. This implies that $K$ is a module over $\End_R \m =\m^*$. So if $1 \in \kappa'(K)$, then $\kappa'$ must be a split epi, contradicting that $K$ is periodic and $\m^*$ is not (Lemma~\ref{why-edim-ge-3}). So $1 \notin \kappa'(K)$, while on the other hand, $\kappa''(1)$ must be nonzero since otherwise $\kappa''$ would be zero. 

Let us observe that the latter sentence implies that $\kappa(K) \cap (\tau^{n+1} M)_{i_0}=0$ where $i_0=\min \{i|  (\tau^{n+1} M)_i \neq 0\}$. To check this, note that $a(R)$ is positive by Condition (*), and therefore by the short exact sequence $0\to R \to \m^* \to k(-a) \to 0$ (Lemma~\ref{lem:Rmk}), we have that $(\m^*)_0=R_0$ is a one-dimensional $k$-vector space, and $(\m^*)_i=R_i=0$ for all $i <0$. So indeed we see that since the image of $K$ does not touch the lowest degree of $\m^*$, and the lowest degree of $\m^*$ does not go to zero in $\tau^{n+1} M$, therefore $\kappa (K)$ does not touch the lowest degree of $\tau^{n+1} M$. 

Therefore $\tau^n p$ restricted to $(\tau^{n+1} M)_{i_0}$ is injective. But together with~\eqref{eqn:first-observe}, this implies $(\tau^n M)_{i_0} \neq 0$, and this is a contradiction since $\tau^n M$ is elevated, by definition of $\mathcal{N}$. So the claim, that  $\tau^{n+1} p$ is epi for all sufficiently large $n \in \mathcal{N}$, is proved.

 Therefore, using  Lemma~\ref{tau-is-lower} and the fact that   $\stHom_R(\tau M,M) \cong \stHom_R(\tau^{n+1} M, \tau^n M)$ for all $n \in \mathbb{Z}$, we may assume  that $p \colon \tau M \to X$ is itself epi and that  $\tau^{-1} X$ is elevated. 

 By   Lemma~\ref{every-irreducible-map}, the AR sequence beginning with $X$  has the form \begin{center}$\xymatrix@C+1.5pc{ 0 \ar[r] & X \ar[r]^-{f=[f_1, f_2]^T} & M \oplus Z \ar[r]^{g=[g_1, g_2]} & \tau^{-1} X \ar[r]& 0}$, \end{center} where $Z$ may be zero. Moreover, $f_1(X) \subseteq \m M$, since otherwise $f_1 p$ gives a nonzero element of $\stHom_R(\tau M, M)_0$, by Lemma~\ref{stablezero-in-radical}. In particular, $f_1$ is mono by Lemma~\ref{irreducible-maps-are-cancellative}.
The surjectivity of $p$ also gives $e(\tau M) > e(X)$, where $e(\,)$ denotes multiplicity, $e(N):=\lim_{n \rightarrow \infty} \frac{1}{n} \length(N/\m^n N)$. (Recall that $e(\,)$ is additive along short exact sequences, and positive on $\CM(R)$.) Therefore, $\tau g_1 \colon \tau M \to X$ is also epi, by Lemma~\ref{irreducible-maps-are-cancellative}. This implies that $\Omega g$ is epi. 

Now let $0 \neq w \in X_{j_0}$ where $j_0=\min\{j \in \mathbb{Z}| X_j \neq 0\}$. Then $f(w) \notin \m M \oplus \m Z$, by Lemma~\ref{lem:counting-k-dim}, and therefore $f_2(w) \neq 0$, since $f_1(X) \subseteq \m M$. Meanwhile   $g_2$ is mono by Lemma~\ref{epi-pairs}, and so $g_2f_2(w) \neq 0$. This contradicts the assumption that $\tau^{-1} X$ is elevated. So the assertion $\stHom_R(\tau M,M)_0 \neq 0$ is proved. Now by Lemma~\ref{lem:in-turn}, $\Ext_R^1(M,M) \neq 0$, i.e. $M$ satisfies the Huneke-Wiegand Conjecture.
\end{proof}

\section{Numerical Semigroups} \label{sec:ns}

The purpose of this section is to prove Proposition~\ref{prop:ns}.

\begin{definition} Let $\mathbb{ N}$ denote the natural numbers $\{0,1,2,\dots\}$. A \emph{numerical semigroup} is a subset $S$ of $\NN$ such that $0 \in S$, $s+s' \in S$ whenever $s$ and $s'$ are in $S$, and $\NN \setminus S$ is finite. The numerical semigroup generated by positive integers $a_1,\dots, a_e$ (such that $\gcd(a_1,\dots, a_e)=1$) is the set $\{n_1a_1+\dots+n_ea_e|n_1,\dots,n_e \in \mathbb{N}\}=\NN a_1+\dots \NN a_e$. If $S= \NN a_1+\dots \NN a_e$, $k[S]$ will denote the graded $k$-algebra $k[t^{a_1},\dots, {t^{a_e}}] \subseteq k[t]$ where $\deg t=1$. The \emph{Frobenius} of a numerical semigroup $S$ is $\max \{n \in \mathbb{Z}| n \notin S\}$, and we denote this integer by $F(S)$. A numerical semigroup $S$ is said to be \emph{complete intersection} if $k[S]$ is a complete intersection. In this case $F(S)=a(k[S])$ (see, e.g., \cite[Proposition 2.3.3]{THESIS}), so in particular we can use formula~\eqref{a(R)-eqn} to find $F(S)$. \end{definition}

\begin{definition} Consider a numerical semigroup $S$ which is minimally generated by $a_1, \dots, a_e$. We may take an additive surjection from the free monoid  $\Free(Y_1,\dots,Y_e)=\{n_1 Y_1 + \dots + n_e Y_e| n_1,\dots,n_e \in \NN\}$, $\varphi \colon \Free(Y_1, \dots Y_e) \to S$ sending $Y_i$ to $a_i$.  A \emph{minimal presentation} for $S$ is a set $\rho \subset \Free(Y_1,\dots,Y_e) \times \Free(Y_1,\dots,Y_e)$ which generates the kernel congruence $\{(u,v) \in \Free(Y_1,\dots,Y_e)  | \varphi(u)=\varphi(v)\}$; see \cite[Chapter 7]{Rosales-Garcia-Sanchez}. For $r=(r_1,r_2) \in \rho$, we will use the notation $|r|=\varphi(r_1)=\varphi(r_2)$. The numbers $\{|r|\}_{r \in \rho}$ are the same as the $\{d_i\}$ from Section~\ref{sec:HW}, when $R \cong k[S]$.\end{definition}

\begin{definition} For a complete intersection numerical semigroup $S$, we will abuse notation slightly by saying that ``$S$ satisfies Condition (*)'' if $k[S]$ satisfies Condition (*). Instead of writing $a(R) > d_{max}/2$, we may alternatively write
 \begin{center}$F(S) > |r|/2$ for all $r \in \rho$, \end{center} where $F(S)$ is the Frobenius of $S$ and $\rho$ is a minimal presentation of $S$. %The reader may simply think of this as a slight short-hand, or may see, e.g., \cite[Chapter 7]{Rosales-Garcia-Sanchez} for the definition of presentation. 
 \end{definition}

\begin{definition}\label{defn:glue} \cite[Ch. 8, \S 3]{Rosales-Garcia-Sanchez} Let $S$ and $S'$ be two numerical semigroups minimally generated by $\{a_1,\dots, a_e\}$ and $\{a'_1,\dots, a'_{e'}\}$, respectively. Let $\lambda \in S \setminus \{a_1,\dots, a_e\}$ and $\mu \in S' \setminus \{a'_1,\dots, a'_{e'}\}$ be such that $\gcd(\lambda, \mu)=1$. Then $S''=\NN \mu a_1+ \dots + \NN  \mu a_e + \NN \lambda a'_1+ \dots + \NN \lambda a'_{e'}$ is called a \emph{gluing} of $S$ and $S'$. \end{definition} 

\begin{lemma} \label{min-gens-of-gluing} In the notation of Definition~\ref{defn:glue}, $S''$ is minimally generated by $\{\mu a_1,\dots, \mu a_e, \lambda a'_1, \dots, \lambda a'_{e'}\}$, and we have $\lambda \ge 2$ and $\mu \ge 2$. \end{lemma}

\begin{proof} For the first statement, see \cite[Lemma 9.8]{Rosales-Garcia-Sanchez}. We have for example $\lambda  \ge 2$ because $\lambda \in S \setminus \{a_1,\dots, a_e\}$. \end{proof}

\begin{lemma}\label{single-additional} (See the proof of \cite[Proposition 9.9]{Rosales-Garcia-Sanchez}.) Assume we have $S$, $S'$, and $S''$ as in Definition~\ref{defn:glue}, and let  $\rho$ and $\rho'$ be minimal presentations of $S$ and $S'$, respectively. Assume that $S$ and $S'$ are complete intersections. Then $S''$ is minimally presented by $\rho \cup \rho'$ together with a single additional element $r$, with $|r| \in \lambda \mu \mathbb{N}$. \end{lemma}

\begin{theorem} \label{glue-theorem} \cite[Theorem 9.10]{Rosales-Garcia-Sanchez} A numerical semigroup other than $\mathbb{N}$ is a complete intersection if and only if it is a gluing of two complete intersection numerical semigroups. \end{theorem}

\begin{lemma}\label{small-exceptions} Let  $m$ and $n$ be relatively prime integers $\ge 2$. Then $mn/2-m-n\ge -2$, and the only cases when $mn/2-m-n\le 0$ are when  either $2 \in \{m,n\}$ or $\{m,n\} \in\{ \{3,4\},\{3,5\}\}$.  \end{lemma}

\begin{lemma} \label{inductive-step} Let $S$ and $S'$ be complete intersection numerical semigroups. Assume that $S$ satisfies Condition (*) and is not equal to $\NN$, and either: $S'$ can be generated by 1 or 2 elements,  or $S'$ satisfies Condition (*). Then each gluing of $S$ and $S'$ satisfies Condition (*) as well. \end{lemma}

\begin{proof}
Assume all notation in Definition~\ref{defn:glue}. Let $\rho$ and $\rho'$ denote  minimal presentations of $S$ and $S'$ respectively, and let  $\{d_1, \dots, d_{e-1}\}=\{|r|\}_{r \in \rho}$ and $\{d'_1,\dots, d'_{e'-1}\}=\{|r'|\}_{r' \in \rho'}$. Formula~\eqref{a(R)-eqn} together with Condition (*) on $S$ say    \begin{equation} \label{star-on-S}
F(S)=\sum _{i=1}^{e-1} d_i -\sum_{i=1}^e {a_i} > d_j/2 \end{equation} for each $j \in \{1,\dots, e\}$. Let $\rho''$ denote a minimal presentation for $S''$. By  Lemma~\ref{single-additional}, $\{|r''|\}_{r'' \in \rho''}=\{\mu d_1,\dots, \mu d_e, \lambda d'_1, \dots, \lambda d'_e, d''\}$, where $d'' \in\lambda \mu \mathbb{N}$, and we have (using~\eqref{a(R)-eqn} and Lemma~\ref{min-gens-of-gluing})\begin{equation} F(S'')=d''+ \sum _{i=1}^{e-1} \mu  d_i +\sum _{i=1}^{e'-1} \lambda d'_i- \sum_{i=1}^{e} \mu {a_i} -\sum_{i=1}^{e'} \lambda {a'_i} =d''+\mu F(S)+\lambda F(S').\end{equation}
Since $F(S) \ge 1$ and $F(S') \ge -1$, we obtain $F(S'')-d''/2\ge d''/2+\mu -\lambda\ge \lambda \mu /2 +\mu -\lambda$, which is positive since $\mu \ge 2$.  Using \eqref{star-on-S} and~\eqref{star-on-S}, we have $F(S'')=d''+\mu F(S)+\lambda F(S')>d''+\mu d_j/2+\lambda F(S') >\mu d_j/2$ for each $j \in \{1, \dots, e-1\}$; and  if $S'$  satisfies Condition (*) then symmetrically $F(S'') > \lambda d'_j/2$ for each $j \in \{1,\dots, e'-1\}$. It remains to check that $F(S'') > \lambda d'_1/2$ whenever $e'=2$. In this case $d'_1=a'_1a'_2$ and the goal amounts to $\lambda (a'_1a'_2/2-a'_1-a'_2)+d''+\mu F(S)>0$. But indeed, $\lambda (a'_1a'_2/2-a'_1-a'_2) \ge -2\lambda$ (by Lemma~\ref{small-exceptions}), $d'' \ge 2 \lambda$, and $F(S)>0$.

\end{proof}

 \begin{proposition}\label{prop:ns} Let $S$ be a complete intersection numerical semigroup, minimally generated by $\{a_1, \dots, a_e\}$. Then $S$ satisfies Condition (*), unless $e=2$ and either $2 \in \{a_1,a_2\}$ or $\{a_1,a_2\} \in\{ \{3,4\},\{3,5\}\}$. \end{proposition}

\begin{proof}  First suppose $e=2$. Now  $d_1=a_1a_2$ and $F(S)=a_1a_2-a_1-a_2$, and we are done by Lemma~\ref{small-exceptions}. 

Now suppose $e=3$. Then, by \cite[Theorem 10.6]{Rosales-Garcia-Sanchez} (or alternatively Theorem~\ref{glue-theorem})  there exist relatively prime integers $m_1$ and $m_2$ greater than one, and nonnegative integers $a$, $b$, and $c$ such that $S=\NN am_1 + \NN am_2 +\NN( bm_1+cm_2 )$, and furthermore, $a \ge 2$, $b+c \ge 2$, and $\gcd(a,bm_1+cm_2)=1$. In this case, the minimal presentation $\rho$ of $S$ has $\{|r|\}_{r \in \rho}=\{am_1m_2, a(bm_1+cm_2)\}$, and \begin{equation} F(S)=a(m_1m_2-m_1-m_2)+(a-1)(bm_1+cm_2). \end{equation}
 First, notice that $F(S)-a(bm_1+cm_2)/2= a(m_1m_2-m_1-m_2)+(a/2-1)(bm_1+cm_2)$ is positive since $m_1m_2-m_1-m_2>0$ and $a \ge 2$. 

To finish the $e=3$ case, suppose (with the aim of a contradiction) $F(S) \le am_1m_2/2$, i.e.,
\begin{equation} \label{eqn:am1m2/2}
a(m_1m_2/2-m_1-m_2) \le (1-a)(bm_1+cm_2). 
\end{equation}

However, $a(m_1m_2/2-m_1-m_2)\ge -2a$ by Lemma~\ref{small-exceptions}; so we get $-2a \le (1-a)(bm_1+cm_2)$, and thus $bm_1+cm_2 \le 2a/(a-1)\le 4$. Since a strict inequality $bm_1+cm_2<4$ is impossible (as $b+c \ge 2$), we get $bm_1+cm_2 = 2a/(a-1)= 4$. So    $a=2$ and  $bm_1+cm_2=4$, contradicting  $\gcd(a,bm_1+cm_2)=1$.

Having finished cases $e=2$ and $e=3$, we proceed to the induction part, which requires little additional work in light of  Theorem~\ref{glue-theorem} and Lemma~\ref{inductive-step}. In fact the only thing left to check is that any gluing of two-generated numerical semigroups (not counting $\mathbb{N}$ as two-generated) satisfies Condition (*). So let $S=\NN a_1 + \NN a_2$, $S'=\NN a'_1 + \NN a'_2 $, $\lambda \in S \setminus \{a_1,  a_2\}$ and $\mu \in S' \setminus \{a'_1, a'_{2}\}$, such that $\gcd(\lambda, \mu)=1$, and let $S''=\NN \mu a_1 + \NN \mu a_2 +\NN \lambda a'_1 + \NN \lambda a'_2$. We wish to check that each of $\{ \mu a_1a_2,\lambda a'_1a'_2,d''\}$ is exceeded by $2 F(S'')$, where $d'' \ge \mu \lambda$ (recall Lemma~\ref{single-additional}) and $F(S'')=d''+\mu(a_1a_2-a_1-a_2)+\lambda(a'_1a'_2-a'_1-a'_2)=d''+F(S)+F(S')$. We have $F(S'')-d''/2>0$ since $F(S)$ and $F(S')$ are positive. Meanwhile, by Lemma~\ref{small-exceptions}, $F(S'')-\mu a_1a_2/2=d''+\mu(a_1a_2/2-a_1-a_2)+F(S') \ge d'' -2 \mu +F(S') \ge \lambda \mu -2 \mu +1>0$. The proof that $F(S'')-\lambda a'_1a'_2/2>0$ is of course the same.
\end{proof}

\def\cprime{$'$} \def\polhk#1{\setbox0=\hbox{#1}{\ooalign{\hidewidth
  \lower1.5ex\hbox{`}\hidewidth\crcr\unhbox0}}}
  \def\polhk#1{\setbox0=\hbox{#1}{\ooalign{\hidewidth
  \lower1.5ex\hbox{`}\hidewidth\crcr\unhbox0}}}
\providecommand{\bysame}{\leavevmode\hbox to3em{\hrulefill}\thinspace}
\providecommand{\MR}{\relax\ifhmode\unskip\space\fi MR }
\providecommand{\MRhref}[2]{%
  \href{http://www.ams.org/mathscinet-getitem?mr=#1}{#2}
}
\providecommand{\href}[2]{#2}

\end{document}